\documentclass[11pt, reqno]{amsart}

\usepackage{amsmath}
\usepackage{amsfonts}
\usepackage{amssymb}
\usepackage{amsthm}
\usepackage{enumitem}
\usepackage{cleveref}
\usepackage{xcolor}
\usepackage{mathrsfs}
\usepackage{bbm}
\usepackage{comment,url}
\usepackage[OT2,T1]{fontenc}
\usepackage{bm}

\newtheorem{theorem}{Theorem}[section]
\newtheorem{lemma}[theorem]{Lemma}
\newtheorem{proposition}[theorem]{Proposition}

\Crefname{conjecture}{Conjecture}{Conjectures}

\theoremstyle{remark}

\theoremstyle{definition}
\theoremstyle{theorem}
\newtheorem{remark}[theorem]{Remark}

\theoremstyle{plain}

\theoremstyle{plain}

\newcommand{\N}{\mathbb{N}}

\newcommand{\Z}{\mathbb{Z}}
\newcommand{\Q}{\mathbb{Q}}
\newcommand{\R}{\mathbb{R}}
\newcommand{\C}{\mathbb{C}}

\newcommand{\J}{\mathbb{J}}
\renewcommand{\S}{\mathbb{S}}
\newcommand{\F}{\mathcal{F}}

\newcommand{\bbH}{\mathbb{H}}

\newcommand{\calL}{\mathcal{L}}
\newcommand{\calM}{\mathcal{M}}
\newcommand{\calG}{\mathcal{G}}

\newcommand{\cusp}{\operatorname{cusp}}

\newcommand{\GL}{\operatorname{GL}}

\newcommand{\SL}{\operatorname{SL}}

\newcommand{\PSL}{\operatorname{PSL}}
\renewcommand{\H}{\mathbb{H}}

\renewcommand{\Im}{\operatorname{Im}}
\newcommand{\sm}{\setminus}
\newcommand{\leg}[2]{\left(\frac{#1}{#2}\right)}
\newcommand{\calR}{\mathcal{R}}

\newcommand{\pabcd}{\left(\begin{smallmatrix} a & b \\ c & d \end{smallmatrix}\right)}

\newcommand{\pmat}[1]{\left(\begin{smallmatrix}#1\end{smallmatrix}\right)}

\newcommand{\ol}[1]{\overline{#1}}

\newcommand{\calF}{\mathcal{F}}

\newcommand{\z}{\mathfrak{z}}
\newcommand{\bbS}{\mathbb{S}}
\newcommand{\GZ}{\mathbb{J}}
\newcommand{\GZp}[3]{\mathbb{J}_{#1,#2,0}^{#3}}

\renewcommand{\a}{\alpha}
\renewcommand{\b}{\beta}
\newcommand{\g}{\gamma}

\renewcommand{\k}{\kappa}
\renewcommand{\l}{\lambda}
\newcommand{\w}{\omega}
\newcommand{\p}{\varrho}
\newcommand{\s}{\sigma}
\renewcommand{\t}{\tau}
\newcommand{\vf}{\varphi}
\newcommand{\De}{\Delta}
\newcommand{\Ga}{\Gamma}
\newcommand{\Om}{\Omega}
\newcommand{\y}{\mathbbm{y}}
\newcommand{\im}{\operatorname{Im}}

\newcommand{\re}{\operatorname{Re}}
\renewcommand{\pmod}[1]{\ \left( \mathrm{mod} \, #1 \right)}
\newcommand{\Pmod}[1]{\ \, ( \mathrm{mod} \, #1 )}
\newif \ifdetails 
\newcommand{\cny}[2]{c\left(#1;#2\right)}

\newcommand{\pd}[2]{\frac{\partial^{#2}}{\partial #1^{#2}}}

\detailstrue 

\newif\ifdiscs 

\discstrue

\ifdetails
\else
\excludecomment{extradetails}
\excludecomment{extradetailsproof}
\fi

\ifdiscs
\else
\excludecomment{discussion}
\fi

\allowdisplaybreaks

\numberwithin{equation}{section}
\author{Kathrin Bringmann}
\address{University of Cologne, Department of Mathematics and Computer Science, Weyertal 86-90, 50931 Cologne, Germany}
\email{kbringma@math.uni-koeln.de}
\author{Ben Kane}
\address{Department of Mathematics, University of Hong Kong, Pokfulam, Hong Kong}
\email{bkane@hku.hk}
\author{Michael H. Mertens}
\address{University of Cologne, Department of Mathematics and Computer Science, Weyertal 86-90, 50931 Cologne, Germany}
\email{mmertens@math.uni-koeln.de}
\title{Hecke eigenforms for meromorphic cusp forms}
\date{\today}


\begin{document}

\begin{abstract}
	In this paper, we construct Hecke eigenforms for two families of quotient spaces of meromorphic cusp forms on $\SL_2(\Z)$. We show that each quotient space in the first (resp. second family) is isomorphic as a Hecke module to the space $S_{2k}$ (resp. $M_{2k}$) of cusp forms (resp. holomorphic modular forms) of the same weight on $\SL_2(\Z)$.
\end{abstract}
\maketitle

\section{Introduction and statement of results}

Hecke theory for holomorphic modular forms is one of the most fruitful tools for (holomorphic) modular forms. It is therefore natural to seek some kind of Hecke structure also for more general spaces of meromorphic modular forms. For the space $M_{2-2k}^!$ of {\it weakly holomorphic modular forms}, i.e., such meromorphic modular forms whose poles, if any, are supported at cusps, this has been achieved by Guerzhoy in \cite{Guerzhoy}. The main difference to the holomorphic theory is that Hecke operators increase the pole order at $i\infty$, so there are no truly weakly holomorphic eigenforms in the strict sense. Guerzhoy \cite{Guerzhoy} instead factored out by a suitable Hecke-stable subspace and defined eigenforms for the quotient space. This subspace is distinguished in a number of ways. For a prime $p$ and $n\in\N$ with $p\mid n$, the $n$-th Fourier coefficients of its elements are divisible by high powers of $p$, yielding congruences for coefficients of weakly holomorphic modular forms \cite{GKO}. Moreover, elements of this space basically lead to vanishing period polynomials \cite{BGKO}. 

As an example, consider the weight $6$ weakly holomorphic modular form
\[
	f_{6,i\infty}(z) := \frac{E_6(z)^3}{\De(z)} + 1488E_6(z) = q^{-1} - 73764q - 86241280q^2 + O\left(q^3\right).
\]
Here and throughout, we let $q:=e^{2\pi iz}$ and define the weight $2k$ {\it Eisenstein series}
\[
	E_{2k}(z):=1-\frac{4k}{B_{2k}} \sum_{n=1}^{\infty} \sigma_{2k-1}(n)q^n,
\]
with $\s_r(n):=\sum_{d\mid n}d^r$ and $B_{k}$ the $k$-th Bernoulli number. Moreover $\De(z):=q\prod_{n=1}^\infty(1-q^n)^{24}=:\sum_{n=1}^\infty\t(n)q^n$ is the unique normalized Hecke eigenform of weight $12$ on $\SL_2(\Z)$. Since the constant term in the Fourier expansion of $f_{6,i\infty}$ vanishes, $f_{6,i\infty}$ belongs to the subspace $S_6^!\subseteq M_6^!$ of {\it weakly holomorphic cusp forms}, i.e., those weakly holomorphic modular forms whose constant term vanishes. Setting $D:=\frac{1}{2\pi i}\pd{z}{}$ and using the fact that $D^{2k-1}$ sends $M_{2-2k}^!$ to $S_{2k}^!$ (see \cite[Theorem 1.2]{BOR}), one verifies that
\[
	D^{5}\left(\frac{E_8(z)}{\Delta(z)}\right) = -f_{6,i\infty}(z).
\]
Since
\[
	\frac{E_8(z)}{\Delta(z)}=q^{-1}+504+73764q + 2695040 q^2+ O\left(q^3\right),
\]
we see that $\frac{E_8}{\De}\notin S_{-4}^!$, so $f_{6,i\infty}\notin D^5(S_{-4}^!)$ because any solution to the differential equation $D^5(F)=-f_{6,i\infty}$ differs from $\frac{E_8}{\De}$ by a polynomial in $z$ of degree at most $5$, and thus $\frac{E_8}{\De}$ is the only modular solution. However, it turns out that for $m\in\N$ we have
\begin{equation}\label{eqn:inftyeigenvalue}
	f_{6,i\infty}\big|T_m -\sigma_5(m) f_{6,i\infty}\in D^{5}\left(S_{-4}^!\right),
\end{equation}
where $T_m$ denotes the $m$-th Hecke operator defined in \eqref{eqn:Tmdef}. 
Denoting 
\begin{equation}\label{eqn:junkinfty}
	\mathbb{J}_{2k}^{i\infty}:=D^{2k-1}\left(M_{2-2k}^!\right)\qquad\text{ and }\qquad \mathbb{J}_{2k,0}^{i\infty}:=D^{2k-1}\left(S_{2-2k}^!\right),
\end{equation}
we see that $f_{6,i\infty}$ has eigenvalue $\sigma_5(m)$ under the Hecke operator $T_m$ in $S_6^!/\mathbb{J}_{6,0}^{i\infty}$. These eigenvalues match those of $E_6$, demonstrating the fact that 
\[
	\left(S_{2k}^!\cap S_{2k}^{\perp}\right)\Big/\mathbb{J}_{2k}^{i\infty}\cong S_{2k}\qquad\text{ and }\qquad \left(S_{2k}^!\cap S_{2k}^{\perp}\right)\Big/\mathbb{J}_{2k,0}^{i\infty}\cong M_{2k}
\]
as Hecke modules (see \cite[Theorem 1.2]{BGKO}).  
Here $S_{2k}^\perp$ denotes the subspace of forms which are orthogonal to $S_{2k}$ under the (suitably regularized) Petersson inner product. In particular, if $k=3$, the quotient space on the left-hand side of the second isomorphism is spanned by $f_{6,i\infty}$, while $M_6$ is spanned by $E_6$ (note that $S_6=\{0\}$, so $S_6^!\cap S_6^\perp=S_6^!$).

Motivated by \eqref{eqn:inftyeigenvalue} and applications such as the Eichler--Shimura theory \cite{BGKO} and congruence properties \cite{GKO}, we construct appropriate subspaces generalizing $\J_{2k}^{i\infty}$ and $\J_{2k,0}^\infty$ to meromorphic modular forms. {\it Meromorphic cusp forms}, originally studied by Petersson \cite{Pe1}, have also been a flourishing subject in the recent past (see e.g. \cite{Bengoechea,BerndtBialekYee,Bialek,BKRamanujanAll}). These forms may have poles in the upper half-plane $\H$ but decay exponentially towards $i\infty$. One of the main problems in finding a notion of Hecke eigenform in this context is that Hecke operators move the pole, which makes it more complicated to find a good space to factor out by. We hence decompose the space of meromorphic cusp forms into certain natural Hecke-stable subspaces and then consider the action of the Hecke operators on these subspaces. Note that for a function $f:\H\to\C$, the space
\[
	\bbS_{2k}(f):=\left\langle f|_{2k}T_m\: :\: m\in\N \right\rangle
\]
is the smallest Hecke-stable subspace containing $f$. We investigate such subspaces for certain natural choices of $f$. Specifically, for fixed $\z\in\H$ and $k,\ell\in\N$ with $k\ge2$, we choose for $f$ the weight $2k$ elliptic Poincar\'e series $\Psi_{2k,-\ell}^\z$, defined in \eqref{eqPsi}. These Poincar\'e series are characterized by their growth at $\z$ and span the subspace of meromorphic cusp forms orthogonal to holomorphic cusp forms by \cite[Satz 7--8]{Pe2} (see also \Cref{lem:PsiGenerate}). We abbreviate
\begin{equation*}
	\bbS_{2k,-\ell}^\z:=\bbS_{2k}\left(\Psi_{2k,-\ell}^{\z}\right)=\left\langle \Psi_{2k,-\ell}^{\z}|_{2k}T_m\: :\: m\in\N \right\rangle.
\end{equation*}
Our generalization of $\J_{2k}^{i\infty}$ and $\J_{2k,0}^{i\infty}$ to subspaces of $\bbS_{2k,-\ell}^{\z}$ is motivated by an algebraicity conjecture of Gross and Zagier \cite[Conjecture (4.4)]{GrossZagier}, which has recently been proven by Li \cite[Theorem 1.1]{Li} and Bruinier--Li--Yang \cite{BLY}. To describe this, suppose that $(\l_m)_{m\ge1}$ is a sequence for which $\l_m\ne0$ for only finitely many $m$. Suppose $\lambda_m$ are chosen such that for any weight $2k$ cusp form $f$ with Fourier expansion $f(z)=\sum_{m=1}^\infty c_f(m)q^m$ 
\begin{equation}\label{eqn:GrossZagierCondition}
	\sum_{m=1}^{\infty} \lambda_m c_f(m)=0.
\end{equation}
Gross and Zagier used such sequences to build linear combinations of the automorphic Green's function\footnote{The automorphic Green's function, defined in \eqref{eqn:Gsdef}, is a non-holomorphic modular form of weight zero which is an eigenfunction under the Laplace operator (see \eqref{eqn:Laplacedef}).} $G_s(z,\z)$ evaluated at different points. They then conjectured that linear combinations constructed from such sequences have certain algebraic properties. For example, a special case of the algebraicity conjecture claims that if $\z$ is a CM-point of fundamental discriminant $D$ and $(\l_m)_{m\ge1}$ satisfies \eqref{eqn:GrossZagierCondition}, then
\begin{equation}\label{eqn:GZconjecture}
	\sum_{m=1}^{\infty} \lambda_m G_k(\z,\z)\big|_0 T_m= u^2D^{1-k}\log|\beta|,
\end{equation}
where $\b$ is an algebraic number lying in the class field of $\Q(\sqrt D)$ and $u\in\N$ is half of the number of roots of unity in $\Q(\sqrt D)$ (see \cite[Conjecture 4.4 on p. 317]{GrossZagier}). 

If $g$ is modular of weight $2\ell-2k$ on $\SL_2(\Z)$, then we define\footnote{Note that the sum does not converge everywhere; see Lemma \ref{lem:Ggcontinue} for details about the meromorphic continuation in $z$.}
\begin{equation}\label{eqn:Ggdef}
	\mathcal{G}_{g}(z,\z)=\mathcal{G}_{2k,-\ell,g}(z,\z):=\sum_{m=1}^{\infty}m^{2k-\ell} g(\z)\big|_{2\ell-2k} T_m q^m. 
\end{equation}
We omit the dependence on $k$ and $\ell$ if it is clear from the context. Fix $k,\ell\in\N$ with $k\ge2$ and $\z\in\H$. Condition \eqref{eqn:GrossZagierCondition} naturally leads one to consider the spaces\footnote{For a detailed explanation of the connection between condition \eqref{eqn:GrossZagierCondition} and \eqref{eqGZ} we refer to \Cref{rem:connection}.}
\begin{align}\label{eqGZ} 
	\GZ_{2k,-\ell}^\z &:= \left\{\mathcal{G}_{g}(z,\z) : g\in R_{2-2k}^{\ell-1}\left(M_{2-2k}^!\right)\right\},\\
	\label{eqGZ2}
	\GZp{2k}{-\ell}{\z} &:= \left\{\mathcal{G}_{g}(z,\z) : g\in R_{2-2k}^{\ell-1}\left(S_{2-2k}^!\right)\right\}.
\end{align}
Here, writing $z=x+iy$ throughout
\[
	R_{2-2k}^{\ell-1}:=R_{2\ell-2k-2}\circ \cdots \circ R_{2-2k}\quad\text{ with }\quad R_{2\kappa}:=2i\frac{\partial}{\partial z}+\frac{2\kappa}{y}
\]
denotes the {\it Maass raising operator} repeated $\ell-1$ times. It raises the weight of a modular object from $2-2k$ to $2\ell-2k$. For $\ell=2k$, Bol's identity (see \cite[Lemma 2.1]{BOR}) yields 
\begin{equation}\label{eqn:BolsIdentity}
	R_{2-2k}^{2k-1}=-(4\pi)^{2k-1} D^{2k-1},
\end{equation}
so these spaces emulate the spaces defined in \eqref{eqn:junkinfty}. We prove in \Cref{thmIso} that $\GZ_{2k,-\ell}^\z$ and $\GZp{2k}{-\ell}{\z}$ are subspaces of $\bbS_{2k,-\ell}^{\z}$. One can thus view the spaces in \eqref{eqGZ} and \eqref{eqGZ2} as images under the map $f\mapsto\calG_{R^{\ell-1}(f)}(z,\z)$ from weakly holomorphic modular forms of weight $2-2k$ to meromorphic cusp forms of weight $2k$. The existence of a weakly holomorphic modular form with a prescribed growth towards $i\infty$ is governed by the Fourier coefficients of weight $2k$ cusp forms (see \cite[Satz 7]{Pe1}), yielding a link with \eqref{eqn:GrossZagierCondition}. 

We next consider an example that emulates the behaviour from \eqref{eqn:inftyeigenvalue}. Set 
\begin{align}\label{eqf6i}
	f_{6,i}(z) &:= \frac{\De(z)}{E_6(z)} = q + 480q^2 + 258804q^3 + 138542080q^4 + O\left(q^5\right),\\
	\nonumber
	g_5(z) &:= \frac1\a\frac{E_8(z)}{\De(z)} \left(j(z)^4\hspace{-.05cm}-\hspace{-.05cm}3480j(z)^3\hspace{-.05cm}+\hspace{-.05cm}3838860j(z)^2 \hspace{-.05cm}-\hspace{-.05cm}1425282400j(z)\hspace{-.05cm}+\hspace{-.05cm}114237825024\right),\\
	\nonumber
	&\hspace{.1cm}= \frac1\a\left(q^{-5}\hspace{-.05cm}-\hspace{-.05cm}3126q^{-1}\hspace{-.05cm}+\hspace{-.05cm} 26994415788736q+519615094283304960q^2\right) \hspace{-.05cm}+\hspace{-.05cm} O\left(q^3\right) \in S_{-4}^!,\\
	\nonumber
	g_7(z) &:= \frac1\a\frac{E_8(z)}{\De(z)}\left(j(z)^6-4968j(z)^5+9176868j(z)^4- 7736486240j(z)^3\right.\\
	\nonumber
	&\hspace{1.25cm}\left.+ 2925506969154j(z)^2-411526489432464j(z)+12317318339088384\right)\\
	\nonumber
	&\hspace{.1cm}= \frac1\a\left(q^{-7}-16808q^{-1}+10625045828793993q+1689691172521357344768q^2\right)\\
	\nonumber
	&\hspace{11cm}+ O\left(q^3\right)\in S^!_{-4},
\end{align}
where $\a:=\frac{E_8(i)}{\De(i)}= 1187.006489\dots$, $j(z):=\frac{E_4(z)^3}{\De(z)}$ is the {\it modular $j$-function}, and the Fourier expansions hold for $y>1$.
A straightforward calculation yields that
\begin{align*}
	f_{6,i}(z)|_6T_5-\sigma_5(5)f_{6,i}(z) &=-2^{-10}\mathcal{G}_{g_5}(z,i)\in\GZp{6}{-1}{i},\\
	f_{6,i}(z)|_6T_7-\sigma_5(7)f_{6,i}(z) &=-2^{-10}\mathcal{G}_{g_7}(z,i)\in\GZp{6}{-1}{i}.
\end{align*}
This behaviour, which parallels \eqref{eqn:inftyeigenvalue}, is not isolated. Much like $f_{6,i\infty}$, the function $f_{6,i}\in\mathbb{S}_{6,-1}^{i}$ generates the one-dimensional quotient space $\mathbb{S}_{6,-1}^{\z}/\GZp{6}{-1}{i}$, and the following theorem gives an isomorphism as Hecke modules between the spaces $S_{2k}$ and $M_{2k}$ and the quotient spaces formed by factoring $\bbS_{2k,-\ell}^{\z}$ out by the subspaces of $\bbS_{2k,-\ell}^{\z}$ defined in \eqref{eqGZ} and \eqref{eqGZ2}. 

\begin{theorem}\label{thmIso}
	For $k,\ell\in\N$ with $k\ge2$, the spaces $\GZ_{2k,-\ell}^\z$ and $\GZp{2k}{-\ell}{\z}$ are Hecke-stable subspaces of $\bbS_{2k,-\ell}^\z$ and the following are isomorphisms of Hecke modules:
	\[
		\bbS_{2k,-\ell}^\z\Big/\GZ_{2k,-\ell}^\z\cong S_{2k},\qquad \bbS_{2k,-\ell}^\z\Big/\GZp{2k}{-\ell}{\z}\cong M_{2k}.
	\]
\end{theorem}

The example $f_{6,i}$ corresponds to $E_6$ under the second isomorphism $\bbS_{2k,-\ell}^\z\big/\GZp{2k}{-\ell}{\z}\cong M_{2k}$ for $k=3$, $\ell=1$, and $\z=i$. We further elaborate upon this in \Cref{sec:example} and also show that
\begin{equation}\label{eqn:defG}
	G(z) := \frac{\De(z)^2}{E_4(z)^3+15^3\De(z)} = q^2 - 4143q^3 + 16868385q^4 - 68686682635q^5 + O\left(q^6\right)
\end{equation}
maps to $\Delta$ under the first isomorphism $\bbS_{2k,-\ell}^\z\big/\GZ_{2k,-\ell}^\z\cong S_{2k}$ for $k=6$, 
$\ell=1$, and $\z=\frac{1+i\sqrt{7}}{2}$.

The paper is organized as follows. In Section \ref{sec:prelim}, we introduce some spaces of non-holomorphic and meromorphic modular forms such as harmonic Maass forms and meromorphic cusp forms, construct bases for these spaces, and recall certain operators. In \Cref{secellprop}, we review some properties of elliptic Poincar\'e series originally considered by Petersson \cite{Pe1,Pe2}. In \Cref{sec:mainthm}, we investigate $\calG_g(z,\z)$ and prove \Cref{thmIso}. Finally, in \Cref{sec:example}, we describe how to construct meromorphic Hecke eigenforms using the isomorphism that is used to prove \Cref{thmIso} and demonstrate this construction by explaining how $f_{6,i}$ is chosen.

\section*{Acknowledgments}

The first author was supported by the Deutsche Forschungsgemeinschaft (DFG) Grant No. BR 4082/5-1. The research of the second author was supported by grants from the Research Grants Council of the Hong Kong SAR, China (project numbers HKU 17303618 and 17314122).

\section{Preliminaries}\label{sec:prelim}

\subsection{Meromorphic modular forms and harmonic Maass forms}\label{sec:modulardef}

We let $\g=\pabcd\in\GL_2^+(\Q)$ act on $\H$ via fractional linear transformations. As usual, for $\k\in\Z$ and a function $f:\H\to\C$, we define the {\it weight $2\k$ slash operator}
\begin{equation}\label{eqn:slashdef}
	f(z)\big|_{2\kappa}\gamma:=\det(\gamma)^{\kappa} (cz+d)^{-2\kappa}f(\gamma z).
\end{equation}
We include the variable $z$ in the notation for the slash operator $|_{2\k,z}$ and other operators if the variable is not clear from context. We say that $f$ satisfies {\it weight $2\k$ modularity for $\SL_2(\Z)$} if $f|_{2\k}\g=f$ for every $\g\in\SL_2(\Z)$. We call $f$ a {\it meromorphic modular form} (for $\SL_2(\Z)$) of weight $2\k$ if $f$ is meromorphic on $\H$, satisfies weight $2\k$ modularity for $\SL_2(\Z)$, and if there exists $n_0\in\Z$ such that $f(z)\asymp e^{-2\pi n_0y}$ as $z\to i\infty$. We say that $f$ has a {\it pole at $i\infty$} if $n_0<0$. A meromorphic modular form is a meromorphic cusp form if $n_0>0$. On the other hand, if $f$ has no poles in $\H$, then $f$ is a {\it weakly holomorphic modular form}. {\it(Holomorphic) modular forms} are those weakly holomorphic modular forms with $n_0\ge0$ and {\it cusp forms} are those holomorphic modular forms with $n_0>0$.

We also require certain non-holomorphic modular forms. Following \cite{BF}, a real-analytic function $f:\H\to\C$ is a {\it harmonic Maass form} of weight $2\k$ (for $\SL_2(\Z)$) if it satisfies:
\begin{enumerate}[leftmargin=*,label=\textnormal{(\arabic*)}]
	\item The function $f$ is modular of weight $2\kappa$ on $\SL_2(\Z)$.
	
	\item The function $f$ is annihilated by the {\it weight $2\kappa$ hyperbolic Laplace operator}
	\begin{equation}\label{eqn:Laplacedef}
		\Delta_{2\kappa}:= -y^2\left(\frac{\partial^2}{\partial x^2}+\frac{\partial^2}{\partial y^2}\right)+2i\kappa y\left(\frac{\partial}{\partial x}+i\frac{\partial}{\partial y}\right).
	\end{equation}

	\item There exists $c\in\R$ such that $f(z)=O(e^{cy})$ as $y\to\infty$.
\end{enumerate}
For $\k<0$ a harmonic Maass form of weight $2\k$ has a Fourier expansion (for some $n_0\in\Z$)
\[
	f(z)=\sum_{n\geq -n_0} c_f^+(n) q^{n} + c_f^-(0) y^{1-2\kappa} +\sum_{\substack{n\leq n_0\\ n\neq 0}} c_f^-(n) \Gamma(1-2\kappa,-4\pi ny) q^{n},
\]
where for $s \in \C$ with $\re(s)>0$ we let $\Gamma(s,y):=\int_{y}^{\infty}t^{s-1}e^{-t}dt$ denote the {\it incomplete gamma function}. The {\it principal part} of $f$ at $i\infty$ is
\[
	\sum_{n<0} c_f^+(n)q^n + c_f^-(0)y^{1-2\k} + \sum_{n>0} c_f^-(n)\Ga(1-2\k,-4\pi ny)q^n.
\]

Defining $\xi_{2\kappa}:=2iy^{2\kappa}\overline{\frac{\partial}{\partial\overline{z}}}$, we have 
\[
	\Delta_{2\kappa}=-\xi_{2-2\kappa}\circ\xi_{2\kappa},
\]
and $\xi_{2\k}$ maps harmonic Maass forms of weight $2\k$ to weakly holomorphic modular forms of weight $2-2\k$ (see \cite[Proposition 3.2]{BF}). We let $H_{2\k}^{\cusp}$ denote the space of harmonic Maass forms which map to cusp forms under $\xi_{2\k}$. Note that a harmonic Maass form $f$ lies in $H_{2\k}^{\cusp}$ if and only if there is a polynomial $P\in\C[q]$ such that $f(z)-P(q)$ is bounded as $y\to\infty$.

\subsection{Operators}

For $m\in\N$, the {\it $V$-operator} is given by 
\begin{equation}\label{eqn:Vdef}
	f(z)\big|V_m:=f(m z)= m^{-\kappa}f(z)\Big|_{2\kappa}\begin{pmatrix}m & 0 \\ 0 & 1\end{pmatrix}.
\end{equation}

We define the \textit{$U$-operator} by
\begin{equation}\label{eqn:Udef}
	f(z)\big|U_m:=m^{\kappa-1}\sum_{j=1}^m f(z)\Big|_{2\kappa}\begin{pmatrix}1 & j \\ 0 & m\end{pmatrix}.
\end{equation}
Note that if $f$ is translation-invariant, then $f$ has a (local) Fourier expansion of the type $f(z)=\sum_{n\in\Z}\cny{n}{y}q^n$. We have
$$
	f(z)\big|V_m=\sum_{n\in\Z}\cny{n}{my}q^{mn}\quad\text{and}\quad f(z)\big|U_m=\sum_{n\in\Z} \cny{mn}{\frac ym}q^n.
$$
Combining these operators, one defines the {\it $m$-th Hecke operator}\footnote{Note that there exist multiple normalizations of $T_m$ (differing by a power of $m$) in the literature. The above definition \eqref{eqn:Tmdef} is chosen so that it preserves integrality of Fourier coefficients for $\k\in\N$.} ($m\in\N$)
\begin{equation}\label{eqn:Tmdef}
	f(z)\big|_{2\kappa}T_m:=\sum_{n\in\Z} \sum_{r\mid \gcd(m,n)} r^{2\kappa-1} \cny{\frac{mn}{r^2}}{\frac{r^2y}{m}} q^n= \sum_{r\mid m} r^{2\kappa-1} f(z)\big|U_{\frac{m}{r}}\circ V_r.
\end{equation}

If $f$ satisfies weight $2\k$ modularity for $\SL_2(\Z)$, then so does $f|_{2\k}T_m$ (see e.g. \cite[Proposition 2.3]{OnoBook}). For $m\in\N$, we let
\begin{gather}\label{eqMm}
	\mathcal{M}_m:=\left\{M\in\textnormal{Mat}_{2\times2}(\Z)\: :\: \det (M)=m\right\}
\end{gather}
be the set of $2\times2$ integer matrices of determinant $m$ and choose a set of representatives 
\[
	\calR_m := \left\{\begin{pmatrix}a&b\\0&d\end{pmatrix} : ad=m,\ d>0,\ b\pmod{d}\right\}
\]
for the right cosets in $\SL_2(\Z)\backslash\mathcal{M}_m$ (see e.g. \cite[p. 38]{Zagier123}).	
\rm
By \eqref{eqn:Tmdef}, \eqref{eqn:Udef}, and \eqref{eqn:Vdef}, we have 
\begin{equation}\label{eqn:Tgamma}
	f(z)\big|_{2\kappa}T_m = m^{\kappa-1} \sum_{M\in \calR_m} f(z)\big|_{2\kappa}M.
\end{equation}

We also require the well-known commutator relations.

\begin{lemma}\label{lem:raiseTm}
Let $f:\bbH\to\C$ be a real-analytic modular form of weight $2\kappa\in2\Z$, i.e., a real-analytic function on $\bbH$ which transforms like a modular form of weight $2\kappa$. Then
\[
R_{2\kappa}\left(f\big|_{2\kappa} T_m\right)= \frac1m R_{2\kappa}(f)\big|_{2\kappa+2}T_m.
\]
\end{lemma}

\begin{proof}
	Plugging the last equalities in \eqref{eqn:Vdef} and \eqref{eqn:Udef} into \eqref{eqn:Tmdef} and using that the raising operator commutes with the slash operator on $\GL_2^+(\R)$ yields the claim.
\end{proof}

\subsection{Elliptic Poincar\'e series}\label{sec:elliptic}

For $z,\z\in\H$, let $X_\z(z):=\frac{z-\z}{z-\ol\z}$, $r_\z(z):=|X_\z(z)|$. For $\k\in\Z$, the weight $2\k$ {\it elliptic expansion around $\z$} of $f:\H\to\C$ is defined as
\[
	f(z)=(z-\overline{\z})^{-2\kappa}\sum_{n\in\Z} c_{\z}\left(r_{\z}(z),n\right) X_{\z}(z)^n,
\]
for $r=r_\z(z)$ sufficiently small.\footnote{Considered by Petersson in \cite[(2a.16)]{Pe1} for holomorphic functions.} If $f$ is meromorphic on $\H$, then there exists $n_{\z,0}$ such that $c_\z(r,n)=0$ for $n<n_{\z,0}$ and $c_{\z}(r,n)$ is independent of $r$; in this case, we simply write $c_\z(n)$. The {\it principal part} of $f$ around $\z$ is given by
\[
	(z-\overline{\z})^{-2\kappa}\sum_{n=n_{\z,0}}^{-1} c_{\z}(n) X_{\z}(z)^n.
\]

For $k\in\N$ with $k\geq 2$ and $\ell\in\Z$, we define the {\it elliptic Poincar\'e series} of weight $2k$ as 
\begin{equation}\label{eqPsi}
	\Psi_{2k,\ell}^\z(z):= \sum_{\gamma\in\SL_2(\Z)} \left.\psi_{2k,\ell}^\z(z)\right|_{2k,z}\gamma,
\end{equation}
with seed 
\begin{equation}\label{eqn:psidef}
	\psi_{2k,\ell}^\z(z):=(z-\overline\z)^{-2k}X_{\z}(z)^\ell.
\end{equation}
Denote by $\w_\z$ half of the size of the stabilizer of $\z$ in $\SL_2(\Z)$. Petersson \cite[Satz 7]{Pe2} showed that if $\ell\not\equiv-k\Pmod{\w_\z}$, then $\Psi_{2k,\ell}^\z$ vanishes identically. Moreover, for $\ell<0$ with $\ell\equiv-k\Pmod{\w_\z}$, \cite[Satz 7]{Pe2} implies that $\Psi_{2k,\ell}^\z$ has principal part $2\w_\z(z-\ol\z)^{-2k}X_\z(z)^\ell$ in its elliptic expansion around the point $\z$ and is holomorphic at points in $\H\cup\{i\infty\}$ which are not $\SL_2(\Z)$-equivalent to $\z$. Moreover, Petersson showed in \cite[S\"atze 7--9]{Pe2} that these Poincar\'e series generate $\S_{2k}$ (see also \cite[Section 4.3]{ImamogluOSullivan}).

\begin{lemma}\label{lem:PsiGenerate}
	The space $\bbS_{2k}$ is generated by $\{\Psi_{2k,\ell}^\z \: : \: \ell\in\Z,\: \z\in\bbH\}$.
\end{lemma}

We also require a two-variable Poincar\'e series that is $\SL_2(\Z)$-invariant, has singularities in $\H$, and is an eigenfunction under the Laplace operator $\De_0$ in both variables. For $z,\z\in\H$ and $s\in\C$ with $\re(s)>1$, the {\it automorphic Green's function} is defined by
\begin{equation}\label{eqn:Gsdef}
	G_s(z,\z) := \sum_{\gamma\in\SL_2(\Z)} g_s^\H (z,\gamma\z).
\end{equation}
Here $g_s^\H(z,\z):=-Q_{s-1}(1+\frac{|z-\z|^2}{2y\y}))$ with $Q_s$ the Legendre function of the second kind (see \cite[14.2.1]{NIST}) and $d(z,\z):=\log(\frac{|z-\ol\z|+|z-\z|}{|z-\ol\z|-|z-\z|})$ denotes the hyperbolic distance between $z$ and $\z$. For further properties of $G_s(z,\z)$, see \cite{Fay,Hejhal}.

\subsection{Harmonic Poincar\'e series}\label{sec:harmonic}

We also require a special family of harmonic Maass forms. To describe it, for $w>0$ set 
\[
	\mathcal{M}_{\kappa,s}(-w):=w^{-\frac{\kappa}{2}} M_{-\frac{\kappa}{2},s-\frac{1}{2}}(w),
\]
where $M_{a,b}(c)$ denotes the $M$-Whittaker function (see \cite[13.14.2]{NIST}). For $k,m\in\N$ define
\[
	\varphi_{2-2k,-m}(z):=-\mathcal{M}_{2-2k,k}(-4\pi m y) e^{-2\pi i mx}.
\]
For $k\ge2$ we let $\calF_{2-2k,-m}$ be the weight $2-2k$ {\it harmonic Maass--Poincar\'e series}
\[
	\calF_{2-2k,-m}(z) := \frac{1}{(2k-1)!}\sum_{\g\in\Ga_\infty\sm\SL_2(\Z)} \vf_{2-2k,-m}(z)\big|_{2-2k}\g,
\]
where $\Ga_\infty:=\{\pm\pmat{1&1\\0&1}:n\in\Z\}$. The function $\calF_{2-2k,-m}$ has principal part $e^{-2\pi imz}$ (see the remark before \cite[Corollary 6.12]{MockBook}) and $\{\calF_{2-2k,-m}:m\in\N\}$ is a basis for the space $H_{2-2k}^{\cusp}$.

\section{Properties of elliptic Poincar\'e series}\label{secellprop}

To prove \Cref{thmIso}, we require some properties of the elliptic Poincar\'e series. We first investigate the relation between the Hecke operators on $\Psi_{2k,\ell}^\z(z)$ in $z$ and $\z$.

\begin{proposition}\label{prop:PsiTnBothVariable}
	For $k\in\N$ with $k\ge2$, $\ell\in\Z$, and $n\in\N$, we have, $\y:=\Im(\z)$ throughout,
	\[
		\Psi_{2k,\ell}^\z(z)\Big|_{2k,z}T_n = \leg n\y^{2k+\ell}\left(\y^{2k+\ell}\Psi_{2k,\ell}^\z(z)\right)\Big|_{-2k-2\ell,\z}T_n.
	\] 
\end{proposition}

In order to prove Proposition \ref{prop:PsiTnBothVariable}, we require another Poincar\'e series.  For $n\in\N$, we choose a set of representatives of the left cosets in $\calM_n/\SL_2(\Z)$ (see \eqref{eqMm}),
\[
	\calL_n := \left\{\begin{pmatrix}a&b\\0&d\end{pmatrix} : ad=n,\ d>0,\ b\pmod{a}\right\}
\]
and define 
\begin{equation}\label{eqn:PsiTndef}
	\Psi_{2k,\ell,n}^\z (z):= n^{k-1}\sum_{\gamma\in \SL_2(\Z)} \sum_{M\in\calL_n} \psi_{2k,\ell}^{\z}(z)\Big|_{2k,z}M\gamma
\end{equation}
with $\psi_{2k,\ell}^\z$ as in \eqref{eqn:psidef}. It is not hard to see that $\Psi_{2k,\ell,n}^\z$ is well-defined. They can be constructed from the functions $\Psi_{2k,\ell}^{\z}$ as follows.

\begin{lemma}\label{lem:PsinPsiTn}
	We have 
	\[
		\Psi_{2k,\ell,n}^{\z}(z)=\Psi_{2k,\ell}^{\z}(z)\big|_{2k,z}T_n.
	\]
\end{lemma}

\begin{proof}
	Swapping the left coset representatives $\calL_n$ appearing in \eqref{eqn:PsiTndef} with right coset representatives $\calR_n$ and combining with \eqref{eqn:Tgamma} and \eqref{eqPsi} gives the result.
\end{proof}

We  have the following relation between $\Psi_{2k,\ell}^\z$ and $\Psi_{2k,\ell,n}^\z$.

\begin{lemma}\label{lem:PsiTnRel}
	For $k\in\N$ with $k\geq 2$, $\ell\in\Z$, and $n\in\N$, we have
	\[
		n^{2k+\ell}\left(\y^{2k+\ell}\Psi_{2k,\ell}^\z(z)\right)\Big|_{-2k-2\ell,\z} T_n = \y^{2k+\ell}\Psi_{2k,\ell,n}^\z(z).
	\]
\end{lemma}

\begin{proof}
	By \eqref{eqn:Tmdef}, plugging in \eqref{eqn:Udef} and \eqref{eqn:Vdef} to evaluate the $U$- and $V$-operators, we have
	\begin{equation}\label{eqn:Heckezet}
		\left(\y^{2k+\ell}\Psi_{2k,\ell}^\z(z)\right)\big|_{-2k-2\ell,\z}T_n = n^{-2k-\ell-1} \y^{2k+\ell} \sum_{r\mid n} r^{2k} \sum_{j=0}^{\frac nr-1} \Psi_{2k,\ell}^{\frac rn(r\z+j)}(z).
	\end{equation}
	On the other hand, we may rewrite $\y^{2k+\ell}$ times the seed from \eqref{eqn:PsiTndef} as
	\[
		\y^{2k+\ell}n^{k-1}\sum_{M\in \mathcal{L}_{n}}\psi_{2k,\ell}^{\z}(z)\Big|_{2k,z}M = \y^{2k+\ell}\frac1n\sum_{r\mid n} r^{2k} \sum_{j=0}^{r-1} \psi_{2k,\ell}^{\z}\left(\frac{r^2z+rj}{n}\right).
	\]
	We make the change of variables $r\mapsto\frac nr$ and plug in the definition of $\psi_{2k,\ell}^\z$ to obtain 
	\begin{align*}
		&\y^{2k+\ell}\frac1n\sum_{r\mid n}\leg nr^{2k}\sum_{j=0}^{\frac nr-1} \psi_{2k,\ell}^\z\leg{\leg nr^2z-\frac nrj}{n}\\
		&\hspace{.5cm}= \y^{2k+\ell}n^{2k-1}\sum_{r\mid n} r^{-2k}\sum_{j=0}^{\frac nr-1} \left(\frac{nz-rj}{r^2}-\ol\z\right)^{-2k}X_\z\leg{nz-rj}{r^2}^\ell\\
		&\hspace{.5cm}= \y^{2k+\ell}\frac1n\sum_{r\mid n} r^{2k}\sum_{j=0}^{\frac nr-1} \left(z-\frac{r^2\ol\z+rj}{n}\right)^{-2k}X_\frac{r^2\z+rj}{n}(z)^\ell = \y^{2k+\ell}\frac1n\sum_{r\mid n} r^{2k}\sum_{j=0}^{\frac nr-1} \psi_{2k,\ell}^\frac{r^2\z+rj}{n}(z).
	\end{align*}
	Therefore, plugging this into \eqref{eqn:PsiTndef}, using \eqref{eqPsi}, and comparing with \eqref{eqn:Heckezet}, we obtain
	\begin{align*}
		\hspace{-.2cm}\y^{2k+\ell}\Psi_{2k,\ell,n}^\z(z) &= \y^{2k+\ell}\frac1n\sum_{r\mid n} r^{2k}\sum_{j=0}^{\frac nr-1}\sum_{\g\in\SL_2(\Z)} \psi_{2k,\ell}^\frac{r^2\z+rj}{n}(z)\Big|_{2k,z}\g\\
		&= \y^{2k+\ell}\frac1n\sum_{r\mid n} r^{2k}\sum_{j=0}^{\frac nr-1} \Psi_{2k,\ell}^\frac{r^2\z+rj}{n}(z) = n^{2k+\ell}\left(\y^{2k+\ell}\Psi_{2k,\ell}^\z(z)\right) \Big|_{-2k-2\ell,\z}T_n. \qedhere
	\end{align*}
\end{proof}

We are now ready to prove Proposition \ref{prop:PsiTnBothVariable}.

\begin{proof}[Proof of Proposition \ref{prop:PsiTnBothVariable}]
Combining Lemma~\ref{lem:PsinPsiTn} and Lemma \ref{lem:PsiTnRel}, we conclude that 
\[
\Psi_{2k,\ell}^{\z}(z)\Big|_{2k,z} T_n= \Psi_{2k,\ell,n}^{\z}(z)=
\left(\frac{n}{\y}\right)^{2k+\ell} \left( \y^{2k+\ell}\Psi_{2k,\ell}^{\z}(z)\right)\Big|_{-2k-2\ell,\z} T_n. \qedhere
\]
\rm
\end{proof}

We also need the following relation between $\Psi_{2k,-\ell}^\z$ and $\calF_{2-2k,-m}$.

\begin{lemma}\label{lemPsiFourier}
	For $\ell\in\N$ and $y>\max\{\Im(\gamma\z)\: :\:\gamma\in\Gamma\}$, we have the Fourier expansion 
	\begin{equation*}
		\Psi_{2k,-\ell}^\z(z)=\frac{(-1)^{k}2^{3-2k}\pi \y^{\ell-2k}}{(\ell-1)! }\sum_{m=1}^\infty R_{2-2k}^{\ell-1}\!\left(\calF_{2-2k,-m}(\z)\right)q^m.
	\end{equation*}
\end{lemma}

\begin{proof}
	We proceed by induction. The case $\ell=1$ is \cite[(3.10) and (3.11)]{BKRamanujanAll}.
	
	Now assume that the lemma holds for some $\ell\in\N$. We multiply both sides of the lemma by $\y^{2k-\ell}$ and apply raising in weight $2\ell-2k$ in $\z$ (for $y>\im(\g\z)$ for all $\g\in\SL_2(\Z)$), yielding
	\begin{multline*}
		\sum_{\gamma\in\SL_2(\Z)}\left( (z-\overline{\z})^{-2k} R_{2\ell-2k,\z} \left(\y^{2k-\ell} X_{\z}(z)^{-\ell}\right)\right) \Big|_{2k,z}\gamma\\
		=\frac{(-1)^{k} 2^{3-2k} \pi}{(\ell-1)!} \sum_{m=1}^{\infty}R_{2-2k}^{\ell}(\mathcal{F}_{2-2k,-m}(\z)) e^{2\pi i mz}.
	\end{multline*}
	We then compute 
	\begin{equation*}
		R_{2\ell-2k,\z}\left(\y^{2k-\ell}X_\z(z)^{-\ell}\right) = \frac{\y^{2k-\ell-1}\left(z-\ol\z\right)^\ell}{(z-\z)^{\ell+1}} ((2k-\ell)(z-\z)+2i\ell\y+(2\ell-2k)(z-\z)),
	\end{equation*}
	where the first term comes from differentiating $\y^{2k-\ell}$, the second comes from differentiating $(z-\z)^{-\ell}$, and the third comes from the term $\frac{2\ell-2k}{\y}$.
	We then rewrite the right-hand side as
	\[
		\ell \frac{\y^{2k-\ell-1} (z-\overline{\z})^{\ell+1}}{(z-\z)^{\ell+1} }=\ell \y^{2k-\ell-1} X_{\z}(z)^{-\ell-1}.
	\]
	Hence
	\[
		\ell \y^{2k-\ell-1} \Psi_{2k,-\ell-1}^{\z}(z) = \frac{ (-1)^{k} 2^{3-2k} \pi}{(\ell-1)! }\sum_{m=1}^\infty R_{2-2k}^{\ell}\!\left(\calF_{2-2k,-m}(\z)\right)q^m.
	\]
	Dividing both sides by $\ell\y^{2k-\ell-1}$ then yields the claim for $\ell+1$, completing the proof.
\end{proof}

\section{Proof of Theorem \ref{thmIso}}\label{sec:mainthm}

Before proving \Cref{thmIso}, we use \Cref{lemPsiFourier} to extend the definition of $\calG_g(z,\z)$.

\begin{lemma}\label{lem:Ggcontinue}
	Let $k\in\N$ with $k\ge2$. For $g\in R_{2-2k}^{\ell-1}(H_{2-2k}^{\cusp})$ and $\z\in\H$, the function $z\mapsto\calG_g(z,\z)$ converges for $y$ sufficiently large (depending on $\z$) and has a meromorphic continuation to $\H$.
\end{lemma}

\begin{proof}
	Choose $F\in H_{2-2k}^{\operatorname{cusp}}$ such that $g=R_{2-2k}^{\ell-1}(F)$. Since $\{\mathcal{F}_{2-2k,-n}:n\in\N\}$ is a basis (see Subsection~\ref{sec:harmonic}) for the space $H_{2-2k}^{\operatorname{cusp}}$, there exist $c(n)\in\C$ and $n_0\in\N$ such that
	\[
		F=\sum_{n=1}^{n_0} c(n) \mathcal{F}_{2-2k,-n}.
	\]
	By comparing principal parts, we have 
	\begin{equation}\label{eqn:FTm}
		n^{1-2k} \mathcal{F}_{2-2k,-n}=\mathcal{F}_{2-2k,-1}\big|_{2-2k}T_n.
	\end{equation}
	Thus we obtain 
	\[
		F=\sum_{n=1}^{n_0} c(n) n^{2k-1} \mathcal{F}_{2-2k,-1}\big|_{2-2k}T_n.
	\]
	Therefore we obtain
	\[
		g=\sum_{n=1}^{n_0} c(n) n^{2k-1} R_{2-2k}^{\ell-1}\left(\mathcal{F}_{2-2k,-1}\big|_{2-2k}T_n\right).
	\]
	Plugging this into \eqref{eqn:Ggdef}, for $y$ sufficiently large (depending on $\z$, where we use the convergence from Lemma \ref{lemPsiFourier}) we have
	\begin{equation}\label{eqn:gTm} 
		\mathcal{G}_{g}(z,\z)= \sum_{m=1}^{\infty} m^{2k-\ell}\sum_{n=1}^{n_0}c(n) n^{2k-1} R_{2-2k}^{\ell-1}\left(\mathcal{F}_{2-2k,-1}(\z)\big|_{2-2k} T_n\right)\big|_{2\ell-2k} T_m q^m. 
	\end{equation}
	We then use Lemma \ref{lem:raiseTm} $\ell-1$ times to rewrite the right-hand side of \eqref{eqn:gTm} as 
	\[
		\sum_{m=1}^{\infty} m^{2k-\ell}\sum_{n=1}^{n_0}c(n) n^{2k-\ell} R_{2-2k}^{\ell-1}\left(\mathcal{F}_{2-2k,-1}(\z)\right)\big|_{2\ell-2k} T_n \circ T_m q^m. 
	\]
	Since Hecke operators commute (see for instance \cite[equation (42)]{Zagier123}), we can rewrite this as 
	\[
		\sum_{m=1}^{\infty} m^{2k-\ell}\sum_{n=1}^{n_0}c(n) n^{2k-\ell} R_{2-2k}^{\ell-1}\left(\mathcal{F}_{2-2k,-1}(\z)\right)\big|_{2\ell-2k} T_m q^m \big|_{2\ell-2k,\z} T_n. 
	\]
	We next interchange
	\rm
	the sums, 
	which is valid due to the growth conditions of 
	\[
		\left|R_{2-2k}^{\ell-1}\left(\mathcal{F}_{2-2k,-1}(\z)\right)\big|_{2\ell-2k} T_m\right|
	\]
	implied by the convergence of the series in \cite[Theorem 3.1]{Fay}.
	Therefore \eqref{eqn:gTm} becomes
	\begin{gather}\label{eqn:gTm2}
		\calG_g(z,\z) = \sum_{n=1}^{n_0} c(n)n^{2k-\ell}\left(\sum_{m=1}^\infty m^{2k-\ell}R_{2-2k}^{\ell-1}(\calF_{2-2k,-1}(\z))\big|_{2\ell-2k}T_mq^m\right) \Big|_{2\ell-2k,\z}T_n.
	\end{gather}
	We use Lemma \ref{lem:raiseTm} again $\ell-1$ times to rewrite the sum inside the parentheses on the right-hand side of \eqref{eqn:gTm2} as 
	\[
		\sum_{m=1}^{\infty} m^{2k-1} R_{2-2k}^{\ell-1}\left(\mathcal{F}_{2-2k,-1}(\z)\big|_{2-2k} T_m\right) q^m.
	\]
	We next employ \eqref{eqn:FTm} again followed by \Cref{lemPsiFourier} to rewrite this as
	\[
		\sum_{m=1}^{\infty} R_{2-2k}^{\ell-1}\left(\mathcal{F}_{2-2k,-m}(\z)\right) q^m=\frac{(-1)^{k}2^{2k-3} (\ell-1)!}{\pi } \y^{2k-\ell} \Psi_{2k,-\ell}^{\z}(z).
	\]
	Plugging this into \eqref{eqn:gTm2}, we obtain
	\begin{equation}\label{eqn:gsumPsirewrite}
		\mathcal{G}_g(z,\z) = \frac{(-1)^{k}2^{2k-3} (\ell-1)!}{\pi} \sum_{n=1}^{n_0}c(n) n^{2k-\ell} \left(\y^{2k-\ell}\Psi_{2k,-\ell}^{\z}(z)\right)\big|_{2\ell-2k,\z}T_n.
	\end{equation}
	Since $\Psi_{2k,-\ell}^{\z}$ is meromorphic and well-defined for all $z\in\H$ and the sum is finite, we obtain a meromorphic continuation for $\mathcal{G}_{g}(z,\z)$. 
\end{proof}

\rm
To prove \Cref{thmIso},we next show that we have Hecke-stable subspaces of $\bbS_{2k,-\ell}^{\z}$.
\rm

\begin{lemma}\label{lem:HeckeEquivariant}
	The spaces $\GZ_{2k,-\ell}^\z$ and $\GZp{2k}{-\ell}{\z}$ are Hecke-stable subspaces of $\bbS_{2k,-\ell}^\z$.
\end{lemma}

\begin{proof}
	We first show that $\GZ_{2k,-\ell}^\z$ and $\GZp{2k}{-\ell}{\z}$ are subspaces of $\S_{2k,-\ell}^\z$. Recalling \eqref{eqGZ} and \eqref{eqGZ2}, both are vector spaces due to the linearity in $g$ in the definition \eqref{eqn:Ggdef} and the fact that $R_{2-2k}^{\ell-1}(M_{2-2k}^!)$ and $R_{2-2k}^{\ell-1}(S_{2-2k}^!)$ are vector spaces. Noting that $\GZp{2k}{-\ell}{\z}\subseteq\GZ_{2k,-\ell}^\z$ because $R_{2-2k}^{\ell-1}(S_{2-2k}^!)\subseteq R_{2-2k}^{\ell-1}(M_{2-2k}^!)$, it remains to show $\GZ_{2k,-\ell}^\z\subseteq\S_{2k,-\ell}^{\z}$. For this, let $\calG_g(z,\z)$ with $g\in R_{2-2k}^{\ell-1}(M_{2-2k}^!)$. We use \Cref{lem:PsiTnRel} 
	followed by \Cref{lem:PsinPsiTn} 
	to rewrite \eqref{eqn:gsumPsirewrite} as 
	\begin{equation}\label{eqn:GgPsi}
		\mathcal{G}_g(z,\z) =\frac{(-1)^{k}2^{2k-3} (\ell-1)!}{\pi} 	\y^{2k-\ell}\sum_{n=1}^{n_0}c(n)\Psi_{2k,-\ell}^{\z}(z)\big|_{2k,z}T_n.
	\end{equation}
	Thus $\mathcal{G}_g(z,\z) \in \mathbb{S}_{2k,-\ell}^{\z}$, so $\GZp{2k}{-\ell}{\z}\subseteq\GZ_{2k,-\ell}^\z\subseteq \mathbb{S}_{2k,-\ell}^{\z}$ are subspaces.

	\rm
	Finally, to show that the subspaces are Hecke-stable, we apply $|_{2k,z}T_r$ to both sides of \eqref{eqn:GgPsi} to obtain
	\begin{equation}\label{eqn:GgTr}
		\mathcal{G}_g(z,\z)\big|_{2k,z}T_r = \frac{(-1)^{k}2^{2k-3} (\ell-1)!}{\pi} \sum_{n=1}^{n_0}c(n) \y^{2k-\ell}\Psi_{2k,-\ell}^{\z}(z)\big|_{2k,z}T_n\circ T_r.
	\end{equation}
	Since the Hecke operators commute, we may interchange them to obtain 
	\begin{equation}\label{eqn:TnTrflip}
		\y^{2k-\ell}\Psi_{2k,-\ell}^{\z}(z)\big|_{2k,z}T_n\circ T_r=\y^{2k-\ell}\Psi_{2k,-\ell}^{\z}(z)\big|_{2k,z} T_r\circ T_n.
	\end{equation}
	We then use Lemma \ref{lem:PsinPsiTn} followed by Lemma \ref{lem:PsiTnRel} (with $\ell\mapsto -\ell$) to rewrite
	\begin{equation}\label{eqn:PsiTnflip}
		\y^{2k-\ell}\Psi_{2k,-\ell}^\z(z)\big|_{2k,z}T_r = r^{2k-\ell} \left(\y^{2k-\ell}\Psi_{2k,-\ell}^\z(z)\right)\big|_{2\ell-2k,\z}T_r.
	\end{equation}
	Plugging \eqref{eqn:TnTrflip} and \eqref{eqn:PsiTnflip} into the right-hand side of \eqref{eqn:GgTr} yields
	\rm
	\[
		\mathcal{G}_g(z,\z)\big|_{2k,z}T_r=\frac{(-1)^{k}2^{2k-3} (\ell-1)!}{\pi} r^{2k-\ell} \sum_{n=1}^{n_0}c(n) \left(\y^{2k-\ell}\Psi_{2k,-\ell}^{\z}(z)\right)\big|_{2\ell-2k,\z}T_r\big|_{2k,z}T_n.
	\]
	We next interchange the Hecke operator $T_n$ in $z$ inside with $T_r$ in $\z$ to obtain
	\begin{align}
		\nonumber \mathcal{G}_g(z,\z)\big|_{2k,z}T_r &=r^{2k-\ell}\left(\frac{(-1)^{k}2^{2k-3} (\ell-1)!}{\pi} \sum_{n=1}^{n_0}c(n) \left(\y^{2k-\ell}\Psi_{2k,-\ell}^{\z}(z)\right)\big|_{2k,z}T_n\right)\big|_{2\ell-2k,\z}T_r\\
		\label{eqn:GgHecke}&\hspace{-.15cm}=r^{2k-\ell} \mathcal{G}_{g}(z,\z)\big|_{2\ell-2k,\z}T_r,
	\end{align}
	using \eqref{eqn:GgPsi} in the last step.
	\rm
	
	Plugging \eqref{eqn:GgHecke} into the definition \eqref{eqn:Ggdef} and then interchanging Hecke operators in $\z$, we obtain that 
	\begin{equation}\label{eqn:GgHeckesubscript}
		\calG_g(z,\z)\big|_{2k,z}T_r= \sum_{m=1}^\infty (mr)^{2k-\ell} g(\z)\big|_{2\ell-2k}T_r \circ T_m q^m = r^{2k-\ell}\calG_{g|_{2\ell-2k}T_r}(z,\z)
	\end{equation}
	Since $M_{2-2k}^!$ and $S_{2-2k}^!$ are both Hecke-stable, Lemma \ref{lem:raiseTm} implies that $R_{2-2k}^{\ell-1}(M_{2-2k}^!)$ and $R_{2-2k}^{\ell-1}(S_{2-2k}^!)$ are both Hecke-stable as well.
\end{proof}

For the proof of Theorem \ref{thmIso} we require certain isomorphisms.

\begin{lemma}\label{lem:HeckeWeakly}
The following spaces are isomorphic as Hecke modules,
\[
H_{2-2k}^{\operatorname{cusp}}\big\slash M^!_{2-2k}\cong S_{2k},\qquad H_{2-2k}^{\operatorname{cusp}}\big\slash S^!_{2-2k}\cong M_{2k}.
\]
\end{lemma}

\begin{proof}
	By \cite[Theorems 1.1 and 1.2 and Lemma 2.1]{BOR} (also see \Cref{lem:raiseTm} for the normalization), $D^{2k-1}$ is an isomorphism from $H_{2-2k}^{\cusp}$ to $S_{2k}^!\cap S_{2k}^\perp$ which, by \eqref{eqn:BolsIdentity} and \Cref{lem:raiseTm}, essentially commutes with the Hecke operators in the sense that
	\[
		D^{2k-1}\left(f\big|_{2-2k}T_m\right)={m^{1-2k}} D^{2k-1}(f)\big|_{2k}T_m.
	\]
	We then consider the reduction of this map modulo $D^{2k-1}(M_{2-2k}^!)$ (resp. $D^{2k-1}(S_{2-2k}^!)$) and use the homomorphism theorem. Since $D^{2k-1}$ is a bijection from $H_{2-2k}^{\operatorname{cusp}}$ to $S_{2k}^!\cap S_{2k}^{\perp}$, the kernel of the composition of $D^{2k-1}$ followed by the reduction map to $D^{2k-1}(M_{2-2k}^!)$ (resp. $D^{2k-1}(S_{2-2k}!)$) is $M_{2-2k}^!$ (resp. $S_{2-2k}^!$). Therefore
	\begin{align*}
		H_{2-2k}^{\operatorname{cusp}}\big\slash M^!_{2-2k}&\cong \left(S_{2k}^!\cap S_{2k}^{\perp}\right)\big\slash D^{2k-1}\left(M_{2-2k}^!\right),\\
		H_{2-2k}^{\operatorname{cusp}}\big\slash S^!_{2-2k}&\cong \left(S_{2k}^!\cap S_{2k}^{\perp}\right)\big\slash D^{2k-1}\left(S_{2-2k}^!\right).
	\end{align*}
	Now using the isomorphism in \cite[Theorem 1.2]{BGKO} completes the proof.
\end{proof}

We are now ready to prove Theorem \ref{thmIso}.

\begin{proof}[Proof of \Cref{thmIso}]
	In order to prove the isomorphisms, for $k,\ell\in \N$ with $k\geq 2$, we consider the map $\varrho_{2-2k,\ell}^{\z}:H_{2-2k}^{\operatorname{cusp}}\mapsto \mathbb{S}_{2k,-\ell}^{\z}$ given by 
	\begin{equation}\label{eqn:varphidef}
		\varrho_{2-2k,\ell}^{\z}(\mathcal{F})(z):= \mathcal{G}_{R_{2-2k}^{\ell-1}(\mathcal{F})}(z,\z).
	\end{equation}
	We claim that $\p_{2-2k,\ell}^\z$ is a bijection between $H_{2-2k}^{\cusp}$ and $\S_{2k,-\ell}^\z$ and essentially commutes with the Hecke operators in the sense that
	\begin{equation}\label{eqn:varphiTm}
		\varrho_{2-2k,\ell}^{\z}(\mathcal{F})\big|_{2k,z}T_m =m^{2k-1}\varrho_{2-2k,\ell}^{\z}\left(\mathcal{F}\big|_{2-2k}T_m\right).
	\end{equation}
	We first show how the claim follows from this: Note that by definition 
	\[
		\varrho_{2-2k,\ell}^{\z}\left(M_{2-2k}^!\right)=\GZ_{2k,-\ell}^{\z}\qquad\text{ and }\qquad \varrho_{2-2k,\ell}^{\z}\left(S_{2-2k}^!\right)=\GZp{2k}{-\ell}{\z}.
	\]
	Hence if $\varrho_{2-2k,\ell}^{\z}$ is a Hecke-equivariant isomorphism from $H_{2-2k}^{\operatorname{cusp}}$ to $\bbS_{2k,-\ell}^{\z}$, then the quotient map $\varrho_{2-2k,\ell}^{\z}\Pmod{\GZ_{2k,-\ell}^\z}$ (resp. $\varrho_{2-2k,\ell}^{\z}\Pmod{\GZp{2k}{-\ell}{\z}}$) has kernel $M_{2-2k}^!$ (resp. $S_{2-2k}^!$), yielding the Hecke-equivariant isomorphisms
	\begin{equation}\label{eqn:iso}
		\mathbb{S}_{2k,-\ell}^{\z}\big\slash{\GZ}^\z_{2k,-\ell} \cong H_{2-2k}^{\operatorname{cusp}}\big\slash M_{2-2k}^!,\qquad \mathbb{S}_{2k,-\ell}^{\z}\Big\slash \GZp{2k}{-\ell}{\z} \cong H_{2-2k}^{\operatorname{cusp}}\big\slash S_{2-2k}^!.
	\end{equation}
	After establishing \eqref{eqn:iso}, the theorem immediately follows by Lemma \ref{lem:HeckeWeakly}. 

	It therefore remains to show that $\p_{2-2k,\ell}^{\z}$ is an isomorphism and satisfies \eqref{eqn:varphiTm}. Since $g\mapsto\calG_g$ is linear, it suffices to prove that it sends a basis of $H_{2-2k}^{\cusp}$ to a basis of $\S_{2k,-\ell}^\z$. From Subsection \ref{sec:harmonic} we know that $\{\calF_{2-2k,-r}:r\in\N\}$ is a basis for $H_{2-2k}^{\cusp}$, while $\{\Psi_{2k,-\ell}^\z|_{2k}T_r:r\in\N\}$ is a basis for $\S_{2k,-\ell}^\z$ by definition of the space. We claim that
	\begin{equation}\label{eqn:varphiFr}
		\varrho_{2-2k,\ell}^{\z}\left(\mathcal{F}_{2-2k,-r}\right) = \frac{(-1)^k 2^{2k-3}(\ell-1)!}{\pi} \y^{2k-\ell}\Psi_{2k,-\ell}^{\z}(z)\big|_{2k,z}T_r, 
	\end{equation}
	which then immediately implies that $\varrho_{2-2k,\ell}^{\z}$ is an isomorphism. 

	To show \eqref{eqn:varphiFr}, we plug \eqref{eqn:FTm} into Lemma \ref{lemPsiFourier} and then apply Lemma \ref{lem:raiseTm}, yielding that for $y$ sufficiently large we have 
	\begin{equation}\label{eqn:Psiexpand}
		\Psi_{2k,-\ell}^\z(z)=\frac{(-1)^{k} 2^{3-2k} \pi \y^{\ell-2k}}{(\ell-1)! }\sum_{m=1}^\infty m^{2k-\ell}R_{2-2k}^{\ell-1}\!\left(\calF_{2-2k,-1}(\z)\right)\big|_{2\ell-2k} T_mq^m.
	\end{equation}
	Combining \eqref{eqn:PsiTnflip} with \eqref{eqn:Psiexpand} and then interchanging the Hecke operators yields
	\begin{multline}\label{eqn:expansionTflip}
		\y^{2k-\ell}\Psi_{2k,-\ell}^\z(z)\big|_{2k,z}T_r\\
		= \frac{(-1)^{k} 2^{3-2k} \pi}{(\ell-1)! } r^{2k-\ell}\sum_{m=1}^\infty m^{2k-\ell}R_{2-2k}^{\ell-1}\!\left(\calF_{2-2k,-1}(\z)\right)\big|_{2\ell-2k}T_r\big|_{2\ell-2k} T_mq^m.
	\end{multline}
	We then use \Cref{lem:raiseTm} $\ell-1$ times to rewrite
	\[
		R_{2-2k}^{\ell-1}\!\left(\calF_{2-2k,-1}(\z)\right)\big|_{2\ell-2k}T_r= r^{\ell-1} R_{2-2k}^{\ell-1}\!\left(\calF_{2-2k,-1}(\z)\big|_{2-2k}T_r\right).
	\]
	The right-hand side of \eqref{eqn:expansionTflip} then becomes
	\begin{equation*}
		\frac{(-1)^{k} 2^{3-2k} \pi r^{2k-1}}{(\ell-1)! }\sum_{m=1}^\infty m^{2k-\ell}R_{2-2k}^{\ell-1}\!\left(\calF_{2-2k,-1}(\z)\big|_{2-2k}T_r\right)\big|_{2\ell-2k} T_mq^m.
	\end{equation*}
	We next use \eqref{eqn:FTm} to rewrite
	\[
		r^{2k-1}R_{2-2k}^{\ell-1}\!\left(\calF_{2-2k,-1}(\z)\big|_{2-2k}T_r\right)= R_{2-2k}^{\ell-1}\!\left(\calF_{2-2k,-r}(\z)\right).
	\]
	Plugging back into \eqref{eqn:expansionTflip} gives
	\[
		\y^{2k-\ell}\Psi_{2k,-\ell}^\z(z)\big|_{2k,z}T_r=\frac{(-1)^{k} 2^{3-2k} \pi}{(\ell-1)! }\sum_{m=1}^\infty m^{2k-\ell}R_{2-2k}^{\ell-1}\!\left(\calF_{2-2k,-r}(\z)\right)\big|_{2\ell-2k} T_mq^m.
	\]
	Comparing with the definitions \eqref{eqn:Ggdef} and \eqref{eqn:varphidef}, we obtain \eqref{eqn:varphiFr}.
	It remains to show that \eqref{eqn:varphiTm} holds. Applying \eqref{eqn:GgHeckesubscript} followed by Lemma \ref{lem:raiseTm}, we have
	\[
		\varrho_{2-2k,\ell}^{\z}(\mathcal{F})(z)\big|_{2k,z}T_m = \mathcal{G}_{R_{2-2k}^{\ell-1}(\mathcal{F})\big|_{2\ell-2k}T_m}(z,\z)=m^{2k-1}\varrho_{2-2k,\ell}^{\z}\left(\mathcal{F}\big|_{2-2k}T_m\right)(z),
	\]
	yielding \eqref{eqn:varphiTm}.
\end{proof}

We end this section with a remark about the motivation for our definition and its relation with Gross and Zagier's conjecture \cite[Conjecture 4.4, p. 317]{GrossZagier}.

\begin{remark}\label{rem:connection}
	Let $(\l_r)_r$ be a sequence of complex numbers, only finitely many of which are $\ne0$. Note that
	\[
		\sum_{r=1}^{\infty}\lambda_r r^{2k-1}\mathcal{F}_{2-2k,-1}|_{2-2k}T_r=\sum_{r=1}^{\infty}\lambda_r\mathcal{F}_{2-2k,-r}\in M_{2-2k}^!
	\]
	if and only if there exists a weakly holomorphic modular form with principal part $\sum_{r=1}^\infty\l_rq^{-r}$ at $i\infty$. Applying Riemann--Roch (for example, see \cite[Satz 1]{Pe1}), such a weakly holomorphic modular form exists if and only if for every $f\in S_{2k}$
	\[
		\sum_{r=1}^{\infty} \lambda_r c_f(r)=0.
	\]
	Hence \eqref{eqn:GrossZagierCondition} is equivalent to
	\[
		g=g_{\lambda}:=\sum_{r=1}^{\infty}\lambda_rR_{2-2k}^{\ell-1}\left(\mathcal{F}_{2-2k,-r}\right)\in R_{2-2k}^{\ell-1}\left(M_{2-2k}^!\right), 
	\]
	and \eqref{eqn:GgPsi} constructs from such $g$ a sum $\mathcal{G}_{g}$ of the $\Psi_{2k,-\ell}^{\z}(z)|_{2k,z}T_n$ analogous to the sum appearing in the formula \eqref{eqn:GZconjecture} conjectured by Gross and Zagier, but in higher weight. Since $\mathcal{G}_g\in \GZ_{2k,-\ell}^{\z}$ if and only if $g\in R_{2-2k}^{\ell-1}(M_{2-2k}^!)$, this formed the motivation for the definitions of the subspaces $\GZ_{2k,-\ell}^\z$ and $\GZp{2k}{-\ell}{\z}$.
\end{remark}

\section{Examples of meromorphic Hecke eigenforms}\label{sec:example}

In this section, we consider two examples of meromorphic Hecke eigenforms which are counterparts of holomorphic modular forms under the isomorphism in Theorem \ref{thmIso}. The first, $f_{6,i}$ as defined in \eqref{eqf6i}, has the same eigenvalues as the Eisenstein series $E_6$. The second example, $G$ defined in \eqref{eqn:defG}, has the same eigenvalues as the cusp form $\Delta$.

\subsection{A meromorphic Hecke eigenform associated to an Eisenstein series}

Here, we investigate $f_{6,i}$ given in the introduction and explain how it is defined. Noting that $S_6=\{0\}$ and $M_6=\C E_6$ \Cref{thmIso} and \eqref{eqn:iso} yields that for every $\ell\in\N$
\[
	\{0\}=S_{6}\cong \mathbb{S}_{6,-\ell}^{\z}\big\slash{\GZ}^\z_{6,-\ell} \cong H_{-4}^{\operatorname{cusp}}\big\slash M_{-4}^!,\qquad \C E_6= M_6\cong \mathbb{S}_{6,-\ell}^{\z}\big\slash \GZp{6}{-\ell}{\z} \cong H_{-4}^{\operatorname{cusp}}\big\slash S_{-4}^!. 
\]
Choosing the generator $\mathcal{F}_{-4,-1}$ of $H_{-4}^{\operatorname{cusp}}\big\slash S_{-4}^!$, by \eqref{eqn:varphiTm} we see that $\varrho_{6,\ell}^{\z}(\mathcal{F}_{-4,-1})$ is a Hecke eigenform in $\bbS_{6,-\ell}^{\z}/\GZp{6}{-\ell}{\z}$ with eigenvalues $n^{2k-1}\sigma_{5}(n)$ under $T_n$.
We choose $\ell=1$ and $\z=i$ and explicitly compute $\varrho_{6,1}^{i}(\mathcal{F}_{-4,-1})$. By \eqref{eqn:gsumPsirewrite} we have
\[
	\varrho_{6,1}^{i}\left(\mathcal{F}_{-4,-1}\right)=\mathcal{G}_{\mathcal{F}_{-4,-1}}(z,i)=-\frac{8}{\pi}\Psi_{6,-1}^i(z).
\]
Normalizing yields the eigenform
\[
	\frac{1}{2^7\pi \alpha} \Psi_{6,-1}^i(z)=\frac{\Delta(z)}{E_6(z)}=f_{6,i}(z).
\]
The calculation given in the introduction illustrates that $f_{6,i}$ is a meromorphic Hecke eigenform corresponding to (a constant multiple of) $E_6$ under the isomorphism in \Cref{thmIso}. Specifically, the eigenvalues under $T_m$ for $m\in\{5,7\}$ are directly computed and shown to match the eigenvalues $\s_5(m)$ of $E_6$ under $T_m$.

\subsection{A meromorphic Hecke eigenform associated to a cusp form}

In this example we establish a meromorphic counterpart of Ramanujan's $\De$-function under the isomorphism in Theorem \ref{thmIso}. Let $\z=\frac{1+\sqrt{7}i}{2}$ to be the unique CM point of discriminant $-7$ in the standard fundamental domain. Furthermore let $\Om:=\Om_{-7}$ the Chowla--Selberg period of discriminant $-7$ (for an explicit formula in terms of the $\Ga$-function see e.g. \cite[equation (97)]{Zagier123}). With this we can explicitly evaluate (see \cite[p. 87]{Zagier123}), 
\begin{equation}\label{eqn:evals}
	E_4(\z)=15\Omega^4,\quad E_6(\z)=27\sqrt{7}\Omega^6,\quad \Delta(\z)=-\Omega^{12}, \quad j(\z)=-15^3.
\end{equation}
Together with the valence formula this implies that the weight $12$ modular form 
\[
	F(z) := E_4(z)^3 + 15^3\De(z) = 1 + 4095q + 98280q^2 + 17805060q^3 + O\left(q^4\right) \in M_{12}
\]
has a simple zero in $\z$ and vanishes nowhere else in the fundamental domain. 

We claim that $G=\frac{\Delta^2}{F}$, defined in \eqref{eqn:defG}, is a meromorphic Hecke eigenform whose eigenvalues agree with those of $\De$. For this, we consider the Maass-Poincar\'e series $\F_{-10,-1}$ as a non-trivial representative of the space 
$H_{-10}^{\operatorname{cusp}}/M_{-10}^!\cong S_{12}=\C\De$, where we use the first isomorphism in Lemma \ref{lem:HeckeWeakly}. As in the previous example, the image of this Poincar\'e series under the map $\varrho_{12,1}^\z$ yields a constant multiple of the elliptic Poincar\'e series $\Psi_{12,-1}^\z$. By \cite[Theorem 1.1]{KenDelta}, for any prime $p$ 
$$
	\mathcal F_{-10,-1}|T_p-p^{-11}\tau(p)\mathcal F_{-10,-1}
$$
is a weakly holomorphic modular form.
Set
\[
	g(z) := \frac{E_4(z)^2E_6(z)}{\De(z)^2} = q^{-2} + 24q^{-1} - 196560 - 47709536q - 3688365156q^2 + O\left(q^3\right)\in M_{-10}^!.
\]
We find that 
\begin{align*}
	\mathcal F_{-10,-1}|T_2-2^{-11}\tau(2)\mathcal F_{-10,-1}&=2^{-11}g,\\
	 \mathcal F_{-10,-1}|T_3-3^{-11}\tau(3)\mathcal F_{-10,-1}&=3^{-11}(j-768)g.
\end{align*}
To see this we note that the principal parts are equal in in each case, wherefore the difference of both sides of the equality must be a holomorphic modular form of weight $-10$ and hence $0$.

We may use these to compute the first few Fourier coefficients of $\Psi_{12,-1}^\z$ in terms of $\a:=\F_{-10,-1}(\z)$ and $\b:=g(\z)=\frac{6075\sqrt7}{\Om^{10}}$, where we employ \eqref{eqn:evals} for the last equality.
Namely, by Lemma \ref{lemPsiFourier}, up to a multiplicative constant, we have
\begin{equation}\label{eqn:PsiFourier}
	\Psi_{12,-1}^\z(z) = \a q + (\b-24\a)q^2 + (252\a-4143\b)q^3 + O\left(q^4\right)
\end{equation}
for $y$ sufficiently large.

In order to show that $G$ is indeed a meromorphic Hecke eigenform corresponding to $\De$, we use \Cref{thmIso} and show that $G$ equals (up to a constant multiple of $\De$ itself) the elliptic Poincar\'e series $\Psi_{12,-1}^\z$. Since $\Psi_{12,-1}^\z$ and $G$ in \eqref{eqn:defG} both have weight $12$ and a simple pole in $\z$ and nowhere else in the fundamental domain, it follows that $\frac{\Psi_{12,-1}^\z}{G}$ is a modular function, holomorphic in $\H$ with at most a simple pole at $i\infty$. Comparing Fourier coefficients in \eqref{eqn:defG} and \eqref{eqn:PsiFourier} we find that
\[
	\frac{\Psi_{12,-1}^\z(z)}{G(z)} = \a q^{-1} + (4119\a+\b) + 196884\a q + O\left(q^2\right) = \a(j(z)+3375) + \b = \a\frac{F(z)}{\De(z)} + \b.
\]
It follows therefore that, again up to a constant factor,
$$
	\Psi_{12,-1}^{\z}=\beta G+\alpha\Delta.
$$
Since we can ignore the term $\alpha\Delta$, Theorem~\ref{thmIso} yields our claim that $G$ is a meromorphic Hecke eigenform corresponding to $\Delta$. To illustrate this we can compute
\begin{equation}\label{eqn:GHecke}
	G(z)|T_2 - (-24)G(z) = q + 16868409q^2 + 279687514914333q^3 + O\left(q^4\right),
\end{equation}
\begin{align}\nonumber
	2^{11}g(z)|T_2 &= q^{-4} + 24q^{-2} + 2048q^{-1} - 402751440 - 7553771839488q + O\left(q^2\right)\\
	\label{eqgT2}
	&= g(z)\left(j(z)^2-1512j(z)+374784\right)\\
	\label{eqgT3}
	3^{11}g(z)|T_3 &= q^{-6} + 24q^{-3} - 34820210880 - 27948629556463536q + O\left(q^2\right)\\
	\nonumber
	&\hspace{-.75cm}= g(z)\left(j(z)^4-3000j(z)^3+2784384j(z)^2-842201064j(z)+52796307708\right).
\end{align}
As predicted by \Cref{thmIso}, evaluating the polynomials in $j$ obtained in \eqref{eqgT2} and \eqref{eqgT3} in $\z$ (i.e., $j(\z)=-15^3$), yields exactly the coefficients of $q^2$ and $q^3$ in \eqref{eqn:GHecke}.

\end{document}